\documentclass[11pt,oneside]{article}
\usepackage[centertags]{amsmath}
\usepackage{amssymb}
\usepackage{amsthm}
\usepackage{hyperref}
\newtheorem{thm}{Theorem}[section]
\newtheorem{cor}[thm]{Corollary}

\newtheorem{prop}[thm]{Proposition}
\newtheorem{defn}[thm]{Definition}
\newtheorem{rem}[thm]{Remark}

\numberwithin{equation}{section}

\newcommand{\condexp}[1]{\left|}

\newcommand{\eps}{\epsilon}
\newcommand{\To}{\longrightarrow}

\newcommand{\A}{\mathcal{A}}
\newcommand{\punt}{\mathbf{.}}

\begin{document}
\title{A new family of time-space harmonic polynomials
with respect to L\'evy processes}
\author{E. Di Nardo \footnote{Dipartimento di Matematica e Informatica,
Universit\`a degli Studi della Basilicata, Viale dell'Ateneo
Lucano 10, 85100 Potenza, Italia, elvira.dinardo@unibas.it}, I.
Oliva \footnote {Dipartimento di Matematica, Universit\`a di
Bologna, Piazza di Porta S. Donato 5, 40126 Bologna, Italia,
oliva@dm.unibo.it}}
\date{\today}
\maketitle
\begin{abstract}
By means of a symbolic method, a new family of time-space harmonic polynomials 
with respect to L\'evy processes is given. The coefficients of these polynomials 
involve a formal expression of L\'evy processes by which many identities are stated. 
We show that this family includes classical families of polynomials such as  
Hermite polynomials. Poisson-Charlier polynomials result to be a linear combinations 
of these new polynomials, when they have the property to be time-space harmonic with 
respect to the compensated Poisson process. The more general class of L\'evy-Sheffer 
polynomials is recovered as a linear combination of these new polynomials, when they
are time-space harmonic with respect to L\'evy processes of very general form. 
We show the role played by cumulants of L\'evy processes so that connections 
with boolean and free cumulants are also stated. 
\end{abstract}
\textsf{\textbf{keywords}:
time-space harmonic polynomial, L\'evy process, cumulant, L\'evy-Sheffer polynomial, umbral calculus}
\section{Introduction}
A family of polynomials $\{P(x,t)\}_{t \geq 0}$ is said to be {\it time-space harmonic} with
respect to a stochastic process $\{X_t\}_{t \geq 0}$ if $E[P(X_t,t) \,\, |\; \mathfrak{F}_{s}]
=P(X_s,s),$ for all $s \leq t,$ where $\mathfrak{F}_{s}=\sigma\left( X_\tau : \tau \leq s\right)$
is the natural filtration associated with $\{X_t\}_{t \geq 0}.$ For random walks $\{X_n\}_{n \geq 0},$ Neveu \cite{neveu} characterizes the family of time-space harmonic polynomials as the coefficients of the Taylor expansion
\begin{equation}
\frac{\exp\{z X_n\}}{E[\exp\{z X_n\}]} = \sum_{k \geq 0} R_k(X_n,n)\frac{z^k}{k!}
\label{expmart}
\end{equation}
in some neighborhood of the origin. If $\{X_n\}_{n \geq 0}$ is replaced by a L\'evy process $\{X_t\}_{t \geq 0},$
the left-hand side of (\ref{expmart}) is the so-called Wald's exponential martingale
\cite{Kuchler}. The usefulness of time-space harmonic polynomials with respect to 
L\'evy processes is that the stochastic process $\{P(X_t,t)\}$ is a martingale, whereas
$\{X_t\}$ does not necessarily have this property.

The Wald's exponential martingale is well defined only when the process admits moment generating
function $E[\exp\{z X_t\}]$ in a suitable neighborhood of the
origin. Different authors have tried to overcome this gap by
using other tools.  Sengupta \cite{Seng00} uses a
discretization procedure to extend the results proved by Goswami
and Sengupta in \cite{GS95}. Sol\'e and Utzet \cite{Sole} use Ito's formula
showing that time-space harmonic polynomials with respect to
L\'evy processes are linked to exponential complete Bell 
polynomials \cite{Comtet}. The Wald's exponential martingale
(\ref{expmart}) has been recently reconsidered also in
\cite{Sengupta08}, but without this giving rise to a closed expression
for these polynomials. 

In this paper,  by using a symbolic method, known in the
literature as the {\it classical umbral calculus}, we give a new family of time-space harmonic
polynomials, which could be easily implemented in any symbolic
software, see \cite{dinoli} as example. Thanks to the results in \cite{Sheffer}, we show that this new 
family includes and generalizes the exponential complete Bell polynomials.

The classical umbral calculus we use consists essentially in a moment symbolic calculus, since its basic 
device is to represent an unital sequence
of numbers by a symbol $\alpha,$ named umbra, i.e. to associate
the sequence $1, a_1, a_2, \ldots$ to the sequence $1, \alpha,
\alpha^2, \ldots$ of powers of $\alpha$ through an operator $E$
that looks like the expectation of random variables \cite{SIAM}. The $n$-th
element of the sequence is the $n$-th moment of $\alpha.$ As a
matter of fact, an umbra looks the framework of a random
variable with no reference to any probability space. 

In this paper, we define an operator $E[ \cdot | \alpha]$ that acts like the well-known 
conditional expectation of random variables.  The umbral version of  L\'evy processes we 
propose takes into account their infinite divisible
property. As corollaries, many identities are given on the coefficients of this family of 
time-space harmonic polynomials with respect to random
walks and L\'evy processes.  Moreover, this expression allows us also to emphasize the role played by cumulants
and to include boolean and free cumulants \cite{Bernoulli, free}. As example, we show that this new family 
includes Hermite polynomials which are time-space harmonic with respect to Brownian motion. We prove that Poisson-Charlier polynomials are linear combinations of the introduced polynomials when they are time-space harmonic 
with respect to compensated Poisson processes. We show that the more general class of L\'evy-Sheffer polynomials \cite{schoutens2, ST} are linear combinations of the introduced polynomials and we characterize a general form for the associated L\'evy processes.  

The paper is structured as follows. Section 2 is provided for readers unaware of the
classical umbral calculus. Let us underline that the theory of the classical umbral
calculus has now reached a more advanced level compared to the elements here resumed.
We have chosen to recall terminology, notation and the basic definitions strictly
necessary to deal with the object of this paper. In Section 3 the new
notion of conditional evaluation with respect to umbrae is introduced, which is
the key to characterize time-space harmonic polynomials in terms of umbrae. 
Examples and applications are introduced in Section 4, where special emphasis is devoted to 
the role played by cumulants in the expression of these polynomials. 
\section{The classical umbral calculus}
In the following, terminology, notation and some basic definitions of the
classical umbral calculus are recalled. We skip any proof: the reader
interested in deeper analysis is referred to the papers \cite{Dinardoeurop,
Sheffer}.

Let ${\mathbb R}[x]$ be the ring of polynomials  with real coefficients in the indeterminate $x.$
The classical umbral calculus is a syntax with an alphabet $\A=\{\alpha,\beta,\gamma, \ldots\}$
of elements, called \emph{umbrae}, and a linear functional $E\,:\,{\mathbb R}[x][\A]\To {\mathbb R}[x]$,
called \emph{evaluation}, such that $E[1]=1$ and
$$E[x^n \, \alpha^i \, \beta^j \, \cdots \, \gamma^k ]=x^n \, E[\alpha^i] \, E[\beta^j] \, \cdots \,  E[\gamma^k] 
\quad \hbox{(uncorrelation property)}$$
where $\alpha, \beta, \ldots, \gamma$ are distinct umbrae and $n,i,j, \ldots,k$ are nonnegative integers.

A sequence $\{a_n\}_{n \geq 0} \in {\mathbb R}[x],$ with $a_0=1,$  is \emph{umbrally represented} by an umbra
$\alpha$ if $E[\alpha^n]=a_n,\;$ for all nonnegative integers $n$. Recall that $a_n$ is called the $n$-th
{\it moment} of $\alpha$. An umbra is \textit{scalar} if its moments are elements of ${\mathbb R}$
while it is \textit{polynomial} if its moments are
polynomials of ${\mathbb R}[x].$ Special scalar umbrae are:
\begin{description}
\item[{\it i)}] the \emph{augmentation} umbra $\epsilon,$ such that $E[\epsilon^n]=\delta_{0,n},$\footnote{The symbol $\delta_{i,j}$ denotes the {\it Kronecker delta}, that is
$\delta_{i,j}=1$ if $i=j,$ otherwise $\delta_{i,j}=0.$}
for all nonnegative integers $n;$ 
\item[{\it ii)}] the \emph{unity} umbra $u,$ such that $E[u^n]=1,$ for all nonnegative integers $n;$
\item[{\it iii)}] the \emph{Bell} umbra $\beta,$ whose moments are the Bell numbers;
\item[{\it iv)}] the \emph{singleton} umbra $\chi$ such that $E[\chi]=1$ and
$E[\chi^n]=0$, for all integers $n > 1.$
\end{description}

The core of this moment symbolic calculus consists in the definition of the {\it dot-product} of two umbrae, which 
is fundamental both in the construction of time-space harmonic polynomials and in their applications.  
We recall in short the steps necessary to give this definition. 

First let us remark that in the alphabet $\A$ two (or more)
distinct umbrae may represent the same sequence of moments. More formally, two umbrae $\alpha$ and $\gamma$ are said to be {\it
similar} when $E[\alpha^n]=E[\gamma^n]$ for all nonnegative integers $n,$ in symbols
$\alpha \equiv \gamma.$ Therefore, given a sequence $\{a_n\}_{n \geq 0},$ there are infinitely many distinct, and thus similar umbrae,
representing this sequence. 

Now, define the symbol $n \punt \alpha$ representing $\alpha'+\alpha''+\cdots+\alpha''',$ where $\{\alpha',\alpha'',\ldots,\alpha'''\}$ is a set of $n$ uncorrelated umbrae similar to $\alpha.$ 
The symbol $n \punt \alpha$ is an example of \emph{auxiliary umbra}. In a \emph{saturated} umbral calculus, the auxiliary umbrae are treated as they were elements of $\A$ \cite{SIAM}.  The umbra $n \punt \alpha$ is called the 
{\it dot-product} of the integer $n$ and the umbra $\alpha.$ Its moments are
\cite{Dinardoeurop}:
\begin{equation}
q_i(n)=E[(n \punt \alpha)^i]=\sum_{k=1}^i (n)_k B_{i,k}(a_1, a_2, \ldots, a_{i-k+1}),
\label{(1)}
\end{equation}
where $(n)_k$ is the lower factorial and $B_{i,k}$ are the exponential partial Bell polynomials
\cite{Comtet}.

In (\ref{(1)}), the polynomial $q_i(n)$ is of degree $i$ in $n.$ If the integer $n$ is replaced by
$t \in {\mathbb R},$ in (\ref{(1)}) we have $q_i(t) = \sum_{k=1}^i (t)_k
B_{i,k}(a_1, a_2, \ldots, a_{i-k+1}).$ We denote by $t \punt \alpha$ the auxiliary umbra 
such that $E[(t \punt \alpha)^i] = q_i(t),$ for all nonnegative integers $i$.
The umbra $t \punt \alpha$ is the dot-product of $t$ and $\alpha.$ Among its properties, we just recall
the distributive property:
\begin{equation}
(t + s) \punt \alpha \equiv t \punt \alpha + s \punt \alpha^{\prime}, \quad s,t \in {\mathbb R}
\label{(distributive)}
\end{equation}
where $\alpha^{\prime} \equiv \alpha.$ In particular in (\ref{(1)}) we can replace $n$ with $-t$ by obtaining the auxiliary umbra $-t \punt \alpha$ with the remarkable property 
\begin{equation}
-t \punt \alpha + t \punt \alpha^{\prime} \equiv \eps,
\label{(inverse)}
\end{equation}
where $\alpha^{\prime} \equiv \alpha.$ Due to property (\ref{(inverse)}), the umbra $-t \punt \alpha$ is named the {\it inverse} umbra of $t \punt \alpha.$ \footnote{Since $-t \punt \alpha$ and $t \punt \alpha$ are two distinct symbols, they can be considered uncorrelated, therefore $-t \punt \alpha + t \punt \alpha^{\prime} \equiv 
-t \punt \alpha + t \punt \alpha \equiv \eps.$ When no confusion occurs, we will use this last similarity instead of
(\ref{(inverse)}).} 

Let us consider again the polynomial $q_i(t)$ and suppose to replace $t$ by an umbra $\gamma.$ The
polynomial $q_i(\gamma)$ is an {\it umbral polynomial} in ${\mathbb R}[x][\A],$ with support 
$\hbox{\rm supp} \,(q_i(\gamma))=\{\gamma\}.$ Recall that the support $\hbox{\rm supp} \, (p)$ of an umbral polynomial
$p \in {\mathbb R}[x][\A]$ is the set of all umbrae occurring in it. The {\it dot-product} of $\gamma$ 
and $\alpha$ is the auxiliary umbra $\gamma \punt \alpha$ such that $E[(\gamma \punt \alpha)^i]=E[q_i(\gamma)]$ for all
nonnegative integers $i.$ Two umbral polynomials $p$ and $q$ are said to be \emph{umbrally equivalent} 
if $E[p]=E[q],$ in symbols $p \simeq q.$ Therefore equation (\ref{(1)}), with $n$ replaced by an umbra $\gamma,$ 
can be written as $q_i(\gamma) \simeq (\gamma \punt \alpha)^i \simeq \sum_{k=1}^i (\gamma)_k B_{i,k}(a_1, a_2, \ldots, a_{i-k+1}).$ Special dot-product umbrae are the 
$\alpha$-cumulant umbra $\chi \punt \alpha$ and $\alpha$-partition
umbra $\beta \punt \alpha,$ that we will use later on. Recall that if $\alpha, \gamma$ and $\eta$ are uncorrelated umbrae, then
\begin{equation}
(\alpha + \eta) \punt \gamma \equiv \alpha \punt \gamma + \eta \punt \gamma.
\label{(distributive1)}
\end{equation}

\section{Time-space harmonic polynomials}
Denote by ${\mathcal X}$ the set ${\mathcal X} = \{\alpha\}.$
\begin{defn}\label{condeval1}
The linear operator $E(\;\cdot\; \vline \,\, \alpha):\, {\mathbb R}[x][\A]
\;\longrightarrow\; {\mathbb R}[{\mathcal X}]$
such that
\begin{itemize}
\item[{\it i)}] $E(1 \,\, \vline \,\,\alpha)=1$;
\item[{\it ii)}] $E(x^m \alpha^n \gamma^i \delta^j\cdots \, \,  \vline \,\, \alpha)=x^m \alpha^n
 E[\gamma^i]E[\delta^j]\cdots$ for uncorrelated umbrae
$\alpha, \gamma, \delta, \ldots$ and for nonnegative integers $m,n,i,j,\ldots$
\end{itemize}
is called \emph{conditional evaluation} with respect to $\alpha.$
\end{defn}

In other words, Definition \ref{condeval1} says that the conditional evaluation
with respect to $\alpha$ handles the umbra $\alpha$ as it
was an indeterminate. The proofs of the following Propositions are straightforward taking into account Definition \ref{condeval1}.
\begin{prop} \label{Propunc}
If $\alpha \in \A$ and $p \in {\mathbb R}[x][\A]$ with
$\alpha \not \in \hbox{\rm supp}(p),$ then $E(p  \,\, \vline \,\, \alpha)=E[p].$
\end{prop}
\begin{cor} \label{Propuncor}
If $\alpha \in \A$ and $p \in {\mathbb R}[x][\A],$ then $E[E(p\,\, \vline \,\, \alpha)]=E[p].$
\end{cor}
This last corollary brings to light the parallelism between the conditional evaluation
$E(\;\cdot\; \vline \,\, \alpha)$ and the well-known conditional expectation in probability
theory \cite{Feller}. As it happens in probability theory, the conditional evaluation is an element of
${\mathbb R}[x][\A]$ and, if we take the overall evaluation of $E(p\,\, \vline \,\, \alpha),$
this gives $E[p].$

The conditional evaluation with respect to the auxiliary umbra $n \punt \alpha$ is such that 
$E[(n+1) \punt \alpha  \,\, \vline \,\, n \punt \alpha] =
E(n \punt \alpha + \alpha^{\prime} \,\, \vline \,\, n \punt \alpha) = n \punt \alpha + E[\alpha^{\prime}],$ 
with $\alpha^{\prime}$ an umbra similar to $\alpha.$ By similar arguments, for all nonnegative integers 
$n$ and $m$ we have
$$E( [(n+m) \punt \alpha]^k  \,\, | \,\, n \punt \alpha) =
E([n \punt \alpha + m \punt \alpha^{\prime}]^k \,\, | \,\, n \punt \alpha) = \sum_{j=0}^k \binom{k}{j} (n \punt \alpha)^j E[(m \punt \alpha^{\prime})^{k-j}],$$ and by taking the evaluation of both sides we recover $E( [(n+m) \punt \alpha]^k)$ due to distributive property (\ref{(distributive)}). 
Therefore, for $t \geq 0$ we define the conditional evaluation of $t \punt \alpha$ with respect to the auxiliary umbra $s \punt \alpha,$ with $0 \leq s \leq t$ such as
\begin{equation}
E[(t \punt \alpha)^k  \,\, | \,\, s \punt \alpha] = \sum_{j=0}^k \binom{k}{j} (s \punt \alpha)^j 
E([(t-s) \punt \alpha^{\prime}]^{k-j}).
\label{inverseII}
\end{equation}

\begin{defn}\label{condeval}
Let $\{P(x,t)\} \in {\mathbb R}[x]$ be a family of polynomials indexed by $t \geq 0.$
$P(x,t)$ is said to be a {\rm time-space harmonic polynomial} with respect to the family
of auxiliary umbrae $\{q(t)\}_{t \geq 0}$ if and only if
$E \left[ P(q(t),t) \, \, \vline \, \, q(s) \right] = P(q(s),s)$ for all $0 \leq s \leq t.$
\end{defn}

\begin{thm}\label{UTSH2}
For all nonnegative integers $k,$ the family of polynomials
\begin{equation}
Q_k(x,t)=E[(x - t \punt \alpha)^k] \in {\mathbb R}[x]
\label{(tshumbral)}
\end{equation}
is time-space harmonic\footnote{When no confusion occurs, we will use the notation $x - t \punt \alpha$ to denote the polynomial umbra $- t \punt \alpha + x = x + (- t \punt \alpha).$} with respect to $\{t \punt \alpha\}_{t
\geq 0}.$
\end{thm}
\begin{proof}
From (\ref{(tshumbral)}), by applying the linearity property of the evaluation $E,$ we have
\begin{equation} 
Q_k(x,t)=\sum_{j=0}^k \binom{k}{j} x^{k-j} E[(- t \punt \alpha)^j]
\label{tsh}
\end{equation}
for all nonnegative integers $k.$ Thanks to (\ref{inverseII}) and (\ref{(distributive)}), we have
\begin{eqnarray*}
E\left(Q_k(t \punt \alpha,t) \, \, \vline \, \, s \punt \alpha \right) & = & 
 \sum_{j=0}^k \binom{k}{j} E[(t \punt \alpha)^{k-j}  \, \, | \, \, s \punt \alpha ] \, E[(- t \punt \alpha)^j] \\
& = & \sum_{j=0}^k \binom{k}{j} \left\{ \sum_{i=0}^{k-j} \binom{k-j}{i} (s \punt \alpha)^{i}
E\left(\{(t - s) \punt \alpha^{\prime} \}^{k-j-i} \right)
 \right\} E[(- t \punt \alpha)^j]. 
\end{eqnarray*}   
By suitably rearranging the terms, we have
\begin{eqnarray*}
E\left(Q_k(t \punt \alpha,t) \, \, \vline \, \, s \punt \alpha \right)  
& = & \sum_{j=0}^k \binom{k}{j} (s \punt \alpha)^{j} \left\{ \sum_{i=0}^{k-j} \binom{k-j}{i} 
E\left(\{(t - s) \punt \alpha^{\prime} \}^{k-j-i}  \right) E[(- t \punt \alpha)^i] 
\right\} \\
& = & \sum_{j=0}^k \binom{k}{j} (s \punt \alpha)^{j} E\left(\{- t \punt \alpha + (t - s) \punt \alpha^{\prime}  \}^{k-j} \right)
= \sum_{j=0}^k \binom{k}{j} (s \punt \alpha)^{j} E[(- s \punt \alpha)^{k-j}].
\end{eqnarray*} 
\end{proof}
\begin{cor} \label{AST_coeff}
If $Q_k(x,t) = \sum_{j=0}^{k} q_j^{(k)}(t) \, x^j,$
then 
\begin{enumerate}
\item[{\it i)} ] $q_j^{(k)}(t) = \binom{k}{j}E[(-t\punt\alpha)^{k-j}],$
for $t >0$ and $j=0,1,\ldots,k;$ in particular $q_k^{(k)}(t)=1;$  
\item[{\it ii)}] $q_j^{(k)}(0) = 0$ for $j=0,1,\ldots,k-1$ and $q_k^{(k)}(0)=1.$ 
\end{enumerate}
In particular we have $Q_k(x,0) = x^k$ for all nonnegative integers $k.$
\end{cor}
\begin{proof}
Property {\it i)} follows from (\ref{tsh}). Property {\it ii)} follows by observing that
$q_j^{(k)}(0) = \binom{k}{j}E[(0 \punt \alpha)^{k-j}] = E[\eps^{k-j}]$ for $j=0,1,\ldots,k-1$ and $q_k^{(k)}(0)=E(\eps^0)=1.$
\end{proof}
The sequence of polynomials $\{Q_k(x,t)\}$ is umbrally represented by the polynomial umbra $x - t \punt \alpha.$
We call $x - t \punt \alpha$ the {\it time-space harmonic polynomial umbra} with respect to
$t \punt \alpha.$ 
\begin{prop} \label{Appell}
The time-space harmonic polynomial umbra $x - t \punt \alpha$ is the Appell umbra of $-t \punt \alpha.$
\end{prop}
The result follows by the definition of Appell umbra given in \cite{Sheffer}. In particular Proposition
\ref{Appell} means that the sequence of polynomials $\{Q_k(x,t)\}$ is an Appell sequence, that is
$$\frac{{\rm d}}{{\rm d} x}Q_k(x,t) = k Q_{k-1}(x,t), \qquad \hbox{{\rm for all integers $k \geq 1$.}}$$

Let us observe that every linear combination of polynomials $\{Q_k(x,t)\}_{k \geq 1}$ is a time-space harmonic polynomial with respect to $\{t \punt \alpha\}_{t \geq 0}  .$
\begin{rem}
{\rm Note that the family of auxiliary umbrae $\{t \punt \alpha\}_{t \in I},$ with $I \subset  {\mathbb R}^{+},$
is the umbral counterpart of a stochastic process $\{X_t\}_{t \in I}$ 
such that $E[X_t^k]=E[(t \punt \alpha)^k]$ for all nonnegative integers $k.$ 
This stochastic process is a L\'evy process, see next section and in particular Remark \ref{cumlev} for a parallelism between the
definition of $t \punt \alpha$ and the infinite divisible property of a L\'evy process. Therefore, the 
polynomials $\{Q_k(x,t)\}$ are time-space harmonic with respect to L\'evy processes. If the moments
of $\{X_t\}_{t \in I}$ are defined only up to some finite $m,$ the representation (\ref{(tshumbral)})
still holds up to $m,$ because it involves moments of order less or equal to $m,$ see also next remark.} 
\end{rem}

\begin{rem}
{\rm The \emph{generating function} of an umbra $\alpha$ is the formal power series
$f(\alpha, z)= \sum_{n\geq0} a_n \frac{z^n}{n!} \in {\mathbb R}[x][[z]]$ whose coefficients are 
the moments of the umbra, see \cite{Dinardoeurop} for more details. Formal power series allow 
us to work  with generating functions which do not have a positive radius 
of convergence or having undefined coefficients  \cite{Stanley}.  The generating 
function of the time-space harmonic polynomial umbra $x -t \punt \alpha$ is 
\begin{equation}
f(x - t \punt \alpha, z)=\frac{\exp\{xz\}}{f(\alpha,z)^t}=\sum_{k \geq 0} Q_k(x,t) \frac{z^k}{k!}.
\label{(wald)}
\end{equation}
By replacing $x$ with $t \punt \alpha$ in (\ref{(wald)}), we recover the Wald's exponential martingale  (\ref{expmart}). Equality of
two formal power series is interpreted as the equality of their coefficients, so that $E[R_k(X_t,t)] = E[Q_k(t \punt \alpha,t)].$ 

Also the Wald's identity $\sum_{k \geq 0} E[R_k(X_t,t)] z^k/k!=1$
follows from (\ref{expmart}). Therefore, the sequence $\{E[R_k(X_t,t)]\}_{k \geq 0} \in {\mathbb R}$ 
is umbrally represented by the augmentation umbra $\eps,$ with $f(\eps,z)=1,$ for all $t \geq 0.$ But this 
is exactly what it happens when in the polynomial umbra $x -t \punt \alpha$ we replace 
$x$ with $t \punt \alpha.$  
}
\end{rem}

The following corollary specifies the dependence of the coefficients of
$Q_k(x,t)$ in (\ref{(tshumbral)})  on the umbra $\alpha.$
\begin{cor}
If $\{a_n\}$ is the sequence umbrally represented by the umbra $\alpha$ and 
$\{Q_k(x,t)\}$ are time-space harmonic polynomials with respect to $\{t \punt \alpha\}_{t \geq 0},$ 
then $Q_k(x,t) = \sum_{j,i=0}^{k} c^{(k)}_{i,j} \, t^i \, x^j,$ with
$$c^{(k)}_{i,j} = \binom{k}{j} \sum_{\lambda \vdash k-j} {\rm d}_{\lambda} (-1)^{2l(\lambda)+i} \, s[l(\lambda),i]
\, a_1^{r_1} a_2^{r_2} \cdots $$
where $s[l(\lambda),i]$ are the Stirling numbers of first kind, 
the sum is over all partitions\footnote{Recall that a partition of an integer
$i$ is a sequence $\lambda = (\lambda_1, \lambda_2, \ldots,
\lambda_m),$ where $\lambda_j$ are weakly decreasing positive
integers such that $\sum_{j=1}^{m} \lambda_j = i.$ The integers
$\lambda_j$ are named {\it parts} of $\lambda.$ The {\it length}
of $\lambda$ is the number of its parts and will be indicated by
$l(\lambda).$  A different notation is $\lambda = (1^{r_1},
2^{r_2}, \ldots),$ where $r_j$ is the number of parts of $\lambda$
equal to $j$ and $r_1 + r_2 + \cdots = l(\lambda).$ Note that
$r_j$ is said to be the multiplicity of $j$. We use the classical
notation $\lambda \vdash i$ to denote \lq\lq $\lambda$ is a
partition of $i$\rq\rq.} $\lambda = (1^{r_1}, 2^{r_2},
\ldots)$ of the integer $k-j$ and $d_{\lambda} = i!/(r_1!
r_2! \cdots \, (1!)^{r_1}(2!)^{r_2} \cdots).$
\end{cor}
\begin{proof}
Equation (\ref{(1)}), with $n$ replaced by $t,$ can be restated as
$E[(-t \punt \alpha)^i] = \sum_{\lambda \vdash i} {\rm d}_{\lambda} (t)_{l(\lambda)}
a_1^{r_1} a_2^{r_2} \cdots,$ see \cite{Bernoulli}. In particular we have $E[(- t  \punt \alpha)^{k-j}] = 
 \sum_{\lambda \vdash k-j} {\rm d}_{\lambda} \,
\left(\sum_{i=0}^{l(\lambda)} s[l(\lambda),i] \, (-1)^{i+2 l(\lambda)} \, t^i \right)
\, a_1^{r_1} a_2^{r_2} \cdots,$ where $s[l(\lambda),i]$ is the $i$-th Stirling number of the first kind.
From Corollary \ref{AST_coeff} we have
\begin{equation}
Q_k(x,t) = \sum_{j=0}^k  \binom{k}{j} \left( \sum_{\lambda \vdash k-j}
{\rm d}_{\lambda} \, a_1^{r_1} a_2^{r_2} \cdots \, \sum_{i=0}^{l(\lambda)} (-1)^{i+2 l(\lambda)} \, s[l(\lambda),i] \, t^i \right) x^j
\label{(5)}
\end{equation}
and the result follows by suitably rearranging the terms in (\ref{(5)}) and by observing that $s[l(\lambda),i] = 0$
for $i > l(\lambda).$
\end{proof}
In the following, assume $Q_k(x,t)$ in (\ref{(tshumbral)}) such that $Q_k(x,t)= \sum_{j=0}^k q_j^{(k)}(t) \, x^j$ and denote by $\{a_n\}$ the sequence of moments umbrally represented by the umbra $\alpha$ in (\ref{(tshumbral)}).
\begin{prop} \label{prop1}
We have $q_j^{(k)}(t-1) = \sum_{i=j}^k\binom{i}{j} q_i^{(k)}(t) \, a_{i-j}.$
\end{prop}

\begin{proof}
From Corollary \ref{AST_coeff}, we have
$q_j^{(k)}(t-1) \simeq \binom{k}{j}[-(t-1)\punt\alpha]^{k-j}\simeq\binom{k}{j}(-t\punt\alpha+\alpha^{\prime})^{k-j}
\simeq \binom{k}{j}\sum_{s=0}^{k-j}\binom{k-j}{s}(-t\punt\alpha)^{k-j-s}\alpha^s
\simeq \sum_{i=j}^k\binom{i}{j} \binom{k}{i} (-t\punt\alpha)^{k-i} \,  \alpha^{i-j}.$
The result follows by taking the evaluation $E$ of both sides. 
\end{proof}

\begin{prop}
We have $a_k = q_0^{(k)}(t-1)-\sum_{j=0}^{k-1} q_j^{(k)}(t) \, a_j. $
\end{prop}
\begin{proof}
By using Proposition \ref{prop1}, we have $q_0^{(k)}(t-1) = a_k q_k^{(k)}(t) \, + \, \sum_{j=0}^{k-1} a_j q_j^{(k)}(t).$ We have
$q_0^{(k)}(t-1) - \sum_{j=0}^{k-1} a_j q_j^{(k)}(t) = a_k q_k^{(k)}(t) \, + \,
\sum_{j=0}^{k-1} a_j q_j^{(k)}(t) - \sum_{j=0}^{k-1} a_j q_j^{(k)}(t)
= a_k q_k^{(k)}(t).$ The result follows by observing that $q_k^{(k)}(t)=1.$ 
\end{proof}
\begin{thm} \label{comb}
A polynomial $P(x,t)=\sum_{j=0}^k p_j(t) \, x^j,$ of degree $k$ for all
$t \geq 0,$ is a time-space harmonic
polynomial with respect to $\{t \punt \alpha\}_{t \geq 0}  $ if and only if
\begin{equation}
p_j(t) = \sum_{i=j}^k  \binom{i}{j} \, p_i(0) \, E[(-t\punt\alpha)^{i-j}], \quad
\hbox{for $j=0,\ldots,k.$} \label{prop5}
\end{equation}
\end{thm}
\begin{proof}
Assume $P(x,t)=\sum_{j=0}^k p_j(t) \, x^j$ a polynomial whose coefficients
satisfy (\ref{prop5}). Then we have
\begin{equation}
\sum_{j=0}^k \left\{ \sum_{i=j}^k \binom{i}{j} \, p_i(0) \, E[(-t\punt\alpha)^{i-j}] \right\} \, x^j
= \sum_{j=0}^k p_j(0) \sum_{i=0}^j \binom{j}{i} \, x^i \, E[(-t\punt\alpha)^{j-i}]
= \sum_{j=0}^k p_j(0) \, E[(x-t\punt\alpha)^{j}].
\label{(series)}
\end{equation}
As $P(x,t)$ is a linear combination of $\{Q_k(x,t)\},$ then $P(x,t)$ is a time-space 
harmonic polynomial with respect to $\{t \punt \alpha\}_{t \geq 0}  .$ Viceversa if $P(x,t)=\sum_{j=0}^k p_j(t) \, x^j$ is a time-space harmonic polynomial with respect to $\{t \punt \alpha\}_{t \geq 0}  ,$ then
$P(x,t) = \sum_{i=0}^k c_i E[(x - t \punt \alpha)^i],$ with $\{c_i\} \in {\mathbb R}.$
Therefore, from (\ref{(series)}) and for $j=0,\ldots,k$ we have $p_j(t) = \sum_{i=j}^k
\binom{i}{j} c_i E[(-t\punt\alpha)^{i-j}].$ So (\ref{prop5}) follows by observing that, when $t$
is replaced by $0,$ we have $p_j(0) = \sum_{i=j}^k \binom{i}{j} \,c_i \, E[(-0\punt\alpha)^{i-j}]
= \sum_{i=j}^k \binom{i}{j} \, c_i \, E[\eps^{i-j}] = c_j.$
\end{proof}
\begin{cor} \label{propfin}
If $P(x,t)=\sum_{j=0}^k p_j(t) \, x^j$ is a polynomial of degree $k$ for all
$t \geq 0,$ then there exists an umbra $\alpha$
such that $P(x,t)$ is a time-space harmonic polynomial with respect to $\{t \punt \alpha\}_{t \geq 0}.$
\end{cor}
\section{Applications and examples}
\subsection{Discrete Case}
When the parameter $t$ is replaced by a nonnegative integer $n,$
the coefficients $q_j^{(k)}(n)$ of time-space harmonic polynomials
satisfy further properties thanks to the umbral re\-pre\-sen\-ta\-tion (\ref{(tshumbral)}).
To keep the length of the paper within bounds, we just show some of them. 
\begin{prop}
We have $q_j^{(k)}(n)+\sum_{i=j+1}^k\binom{i}{j}\sum_{l=1}^n q_i^{(k)}(l) \, a_{i-j} = 0.$
\end{prop}
\begin{proof}
From Corollary \ref{AST_coeff}
$\sum_{i=j+1}^k\binom{i}{j}\sum_{l=1}^n\alpha^{i-j} q_i^{(k)}(l)
        \simeq\sum_{l=1}^n\sum_{i=j+1}^k\binom{i}{j}\alpha^{i-j}\binom{k}{i}(-l\punt\alpha)^{k-i}.$
Since
$\sum_{i=j+1}^k\binom{i}{j}\alpha^{i-j}\binom{k}{i}(-l\punt\alpha)^{k-i} \simeq \binom{k}{j}
\left[(-l\punt\alpha + \alpha^{\prime})^{k-j}-(-l\punt\alpha)^{k-j}\right]\simeq q_j^{(k)}(l-1)- q_j^{(k)}(l),$
the result follows by observing that
$ \sum_{i=j+1}^k\binom{i}{j}\sum_{l=1}^n\alpha^{i-j} q_i^{(k)}(l) \simeq \sum_{l=1}^n\left( q_j^{(k)}(l-1)- q_j^{(k)}(l)\right)
\simeq q_j^{(k)}(0) - q_j^{(k)}(n) \simeq - q_j^{(k)}(n),$
since $q_j^{(k)}(0) \simeq \binom{k}{j} \, (-0\punt\alpha)^{k-j} \simeq 0$, for all $k\neq j.$
\end{proof}
\begin{cor}
We have $q_0^{(k)}(n)+\sum_{l=1}^n\sum_{j=1}^k a_j \, q_j^{(k)}(l) = 0.$
\end{cor}

\begin{rem}
{\rm Let us observe that the family of umbrae $\{n \punt \alpha\}_{n\geq 0}$ corresponds to a discrete martingale
$\{X_n\}_{n \geq 0}$ with $X_0=0$ and independent and identically distributed  difference sequence with zero mean.
Recall that the difference sequence associated to $\{X_n\}_{n \geq 0}$ is a sequence of random variables $\{M_n\}_{n \geq 0}$ such that
$M_0=X_0=0$ and $M_n=X_n-X_{n-1},$ for all nonnegative integers $n.$
The umbra $n \punt \alpha$ generalizes $X_n=M_1+M_2+\cdots+M_n.$
Suppose to remove the identical distribution hypothesis on $\{M_n\}_{n \geq 0}:$ in umbral terms, the martingale
$\{X_n\}_{n \geq 0}$ corresponds to the umbra $\alpha_1+\alpha_2+\cdots+\alpha_n,$ where the umbrae
$\{\alpha_1, \ldots, \alpha_n\}$ are not necessarily similar.
The time-space harmonic polynomials $E[(x-n \punt \alpha)^k]$ need to be replaced by $E[\left(x- 1 \punt (\alpha_1+\alpha_2+\cdots+\alpha_n)\right)^k].$ The properties stated up to now can be recovered by similar arguments.}
\end{rem}
\subsection{Cumulants}
Any umbra is a partition umbra (cf. \cite{Dinardoeurop}). This means that 
if $\{a_n\}_{n \geq 0}$ is a sequence umbrally represented by an umbra $\alpha,$
then there exists a sequence $\{h_n\}_{n \geq 1}$ umbrally represented by an umbra
$\kappa_{{\scriptscriptstyle \alpha}},$ such that $\alpha \equiv \beta \punt
\kappa_{{\scriptscriptstyle \alpha}}.$ In terms of generating functions we have $f(\alpha,z)=
\exp[f(\kappa_{{\scriptscriptstyle \alpha}},z)-1],$ so that $\{h_n\}_{n \geq 1}$
is the sequence of formal cumulants of $\{a_n\}_{n \geq 0}$\footnote{In the ring of formal power series, given a sequence $\{a_n\}_{n \geq 1},$
its sequence of formal cumulants $\{h_n\}_{n \geq 1}$ is such that $1+ \sum_{n\geq 1} a_n z^n / n! = \exp \left(
\sum_{n\geq 1} h_n z^n / n!\right).$}. The umbra $\kappa_{{\scriptscriptstyle \alpha}}$ is called $\alpha$-{\it cumulant} umbra
and we also have $\kappa_{{\scriptscriptstyle \alpha}} \equiv \chi \punt \alpha,$
where $\chi$ is the singleton umbra.
\begin{prop}
For the sequence of polynomials $\{Q_k(x,t)\}$ umbrally represented by the time-space harmonic polynomial
umbra $x - t \punt \alpha,$ we have 
\begin{equation}
Q_k(x,t)=Y_k(x+h_1, h_2, \ldots, h_k),
\label{(complBell)}
\end{equation}
with $Y_k$ exponential complete Bell polynomials and $\{h_n\}$ the sequence of cumulants of $-t \punt \alpha.$
\end{prop}
\begin{proof}
We have $E[(\beta \punt \gamma)^k]= Y_k(g_1, g_2, \ldots, g_k)$ with $g_n = E[\gamma^n],$ for all nonnegative integers $n$ \cite{Dinardoeurop}. Therefore we will prove (\ref{(complBell)}), if we show that $x + \punt (-t \punt \alpha) \equiv \beta \punt \gamma$ for some polynomial umbra $\gamma.$ Choose as umbra $\gamma$ 
the umbra $\kappa_{(x \punt u)} \dot{+} \kappa_{(-t \punt \alpha)}$ with $\kappa_{(x \punt u)}$ the cumulant umbra of $x \punt u$ and $\kappa_{(-t \punt \alpha)}$ the cumulant umbra of $- t \punt \alpha.$ 
Let us recall that the disjoint sum of two distinct umbrae $\delta_1 \dot{+} \delta_2$ is an auxiliary umbra whose $n$-th moment is $E[\delta_1^n] + E[\delta_2^n]$ for all integers $n \geq 1.$ We have     
$$E[(\kappa_{(x \punt u)} \dot{+} \kappa_{(-t \punt \alpha)})^n] = \left\{ \begin{array}{ll}
x + h_1 & n=1 \\
h_n & n > 1
\end{array} \right.$$
with $\{h_n\}$ the sequence of cumulants of $-t \punt \alpha.$ The result follows since  $x + t \punt (-1 \punt \alpha) \equiv \beta \punt \kappa_{(x \punt u)} + \beta \punt \kappa_{(-t \punt \alpha}) \equiv \beta \punt (\kappa_{(x \punt u)} \dot{+} \kappa_{(-t \punt \alpha)}),$ \cite{Sheffer}.
\end{proof}

\begin{rem} \label{cumlev}
{\rm As $f(\beta \punt \kappa_{{\scriptscriptstyle \alpha}},z)=
\exp[f(\kappa_{{\scriptscriptstyle \alpha}},z)-1],$ the auxiliary umbra
$\beta \punt \kappa_{{\scriptscriptstyle \alpha}}$ is the umbral counterpart
of a compound Poisson random variable with parameter $1.$ More in general, compound
Poisson random variables of parameter $t$ are represented by the auxiliary
umbrae $t \punt \beta \punt \kappa_{{\scriptscriptstyle \alpha}},$
with generating function $f(t \punt \beta \punt \kappa_{{\scriptscriptstyle \alpha}},z)=\exp[t(f(\kappa_{{\scriptscriptstyle \alpha}},z)-1)]$
and we have $t \punt \alpha \equiv t \punt \beta \punt \kappa_{{\scriptscriptstyle \alpha}}.$ 

Let $\{X_t\}_{t \geq 0}$ be a real-value L\'evy process,
i.e. a process starting from $0$ and with stationary and independent increments.
If we denote the moment generating function of $X_{t+s}-X_s$ by $\phi(z,t),$ then
$\phi(z,t)$ is infinitely divisible \cite{sato} and $\phi(z,t) = (\phi(z,1))^t = (\psi(z))^t,$
where $\psi(z)$ is the moment generating function of $X_1.$ In particular,
$\phi(z,t) = \exp[t\log\psi(z)] = \exp[t k(z)],$ where $k(z)$ is the cumulant
generating function of $X_1$ such that $k(0)=0.$
Comparing $\exp[t k(z)]$ with $\exp[t(f(\kappa_{{\scriptscriptstyle \alpha}},z)-1)],$
the correspondence between $t \punt \alpha \equiv t \punt \beta \punt \kappa_{{\scriptscriptstyle \alpha}}$ 
and the L\'evy process $\{X_t\}_{t \geq 0}$ is immediate. We call $t \punt \alpha$  the 
{\it L\'evy umbra} associated to the umbra $\alpha.$ } 
\end{rem}

The following theorem explicitly states the connection between time-space harmonic polynomials
with respect to L\'evy processes $\{X_t\}_{t \geq 0}$ and the sequence of cumulants of $X_1.$
\begin{thm} \label{AST_Levy1}
For all nonnegative integers $k,$ the family of polynomials
\begin{equation}
Q_k(x,t)=E[(x - t \punt \beta \punt \kappa_{{\scriptscriptstyle \alpha}})^k] \in {\mathbb R}[x]
\label{(tshumbral1)}
\end{equation}
is time-space harmonic with respect to $\{t \punt \alpha\}_{t \geq 0},$ where $\kappa_{{\scriptscriptstyle \alpha}}$ is the $\alpha$-cumulant umbra. 
\end{thm} 
The umbra $-t \punt \alpha \equiv t \punt (-1 \punt \alpha)$ is the L\'evy umbra associated to the umbra $-1 \punt \alpha. $ Therefore, also the polynomials $Q_k(x,t) = E[(x + t \punt \beta \punt \kappa_{{\scriptscriptstyle \alpha}})^k]$ are time-space harmonic with respect to L\'evy processes umbrally represented by $\{-t \punt \alpha\}_{t \geq 0}.$ For simplicity, the following results refer to this last class of polynomials, but they can be stated also for the polynomials given in (\ref{(tshumbral1)}). 

Moments of a polynomial umbra as $x + t \punt \beta \punt \gamma$ are generalizations of exponential complete Bell polynomials, see \cite{Sheffer}.  
\begin{prop}[Sheffer identity with respect to $t$] \label{main}
For the sequence of polynomials $\{Q_k(x,t)\}$ in (\ref{(tshumbral1)}) the following identity holds
$$Q_k(x,t+s)=\sum_{j=0}^k \binom{k}{j} P_j(s) Q_{k-j}(x,t),$$
where $P_j(s)=Q_{j}(0,s)$ for all nonnegative integers $j.$ 
\end{prop}
\begin{proof}
The result follows as $Q_k(x,t+s)  = E\left([x + t \punt \beta \punt 
\kappa_{{\scriptscriptstyle \alpha}} + s \punt \beta \punt
\kappa_{{\scriptscriptstyle \alpha}}]^ k\right)  =  \sum_{j=0}^k \binom{k}{j} E\left[(s \punt \beta \punt
\kappa_{{\scriptscriptstyle \alpha}})^j\right]$ $E\left[(x + t \punt \beta \punt
\kappa_{{\scriptscriptstyle \alpha}})^{k-j}\right].$ 
\end{proof}
\begin{cor}
For the coefficients $\{q_j^{(k)}(t)\}$ of the sequence of polynomials $\{Q_k(x,t)\}$ in (\ref{(tshumbral1)}), 
we have  
$$\frac{{\rm d}}{{\rm d} t} q_j^{(k)}(t) = \sum_{i=1}^{k-j} {k \choose i} h_i q_j^{(k-i)}(t)$$
for $j=1,\ldots,k,$ where $\{h_i\}$ is the sequence of cumulants of $\alpha.$ 
\end{cor}
\begin{proof} 
Consider $t$ as an indeterminate and observe that
$\frac{{\rm d}}{{\rm d} t} q_j^{(k)}(t) = \binom{k}{j} \frac{{\rm d}}{{\rm d} t}E\left[(t\punt\beta\punt
\kappa_{{\scriptscriptstyle \alpha}})^{k-j}\right].$  We have 
\begin{equation}
\frac{{\rm d}}{{\rm d} t} q_j^{(k)}(t) \simeq \binom{k}{j}
\left\{[(t+\chi)\punt\beta\punt\kappa_{{\scriptscriptstyle \alpha}}]^{k-j}-(t\punt\beta
\punt\kappa_{{\scriptscriptstyle \alpha}})^{k-j}\right\} \simeq \binom{k}{j}\sum_{i=1}^{k-j}\binom{k-j}{i}  \kappa_{{\scriptscriptstyle \alpha}}^i(t\punt\beta\punt \kappa_{{\scriptscriptstyle \alpha}})^{k-j-i}, 
\label{(last1)}
\end{equation} 
see \cite{Sheffer} for the first equivalence. The result follows by taking the evaluation of both sides
in (\ref{(last1)}).
\end{proof}
%
\subsection{Special families of polynomials}
In this section we show that some classical families of time-space harmonic 
polynomials are moments of the polynomial umbra $x - t \punt \alpha,$ or can 
be recovered as a linear combination of its moments, for 
a suitable umbra $\alpha.$ 
\subparagraph{Hermite polynomials.}
Standard Brownian motion is a special L\'evy process. From the L\'evy-Khintchine formula \cite{sato}, its generating
function is $\phi(z,t)=\exp(t \, z^2 /2).$  Therefore, the umbral counterpart of a standard Brownian
motion is the umbra $t \punt \beta \punt \delta,$ where $f(\delta,z)=1+z^2/2.$ From Theorem \ref{AST_Levy1},
time-space harmonic polynomials with respect to a standard Brownian motion are $Q_k(x,t) = E[(x-t\punt\beta\punt\delta)^k].$
\begin{prop}
For all nonnegative integers $k$ we have $Q_k(x,t) = H_k^{(t)}(x),$ where $H_k^{(s^2)}(x)$
are the generalized Hermite polynomials with generating function
$\sum_{k \geq 0} H_k^{(s^2)}(x)\frac{z^k}{k!} = \exp\{xz - s^2\frac{z^2}{2}\}.$
\end{prop}
\begin{proof}
We have $H_k^{(s^2)}(x) = E \left[ \left(x - 1 \punt \beta \punt (s\delta)\right)^k \right],$ see 
\cite{dinoli}. The result follows by observing that $-1 \punt \beta \punt (\sqrt{t}\delta) \equiv 
-t \punt \beta \punt \delta.$
\end{proof}
\subparagraph{Poisson-Charlier polynomials.}
It is known that the Poisson-Charlier polynomials $\widetilde{C}_k(x + t,t)$ are time-space harmonic with respect to the compensated Poisson process \cite{Sole}, with $\{\widetilde{C}_k(x,t)\}$ polynomials having generating function $\sum_{k\geq 0}\widetilde{C}_k(x,t)\frac{z^k}{k!}=e^{-tz}(1+z)^x.$ We will recover this result, by proving that $\widetilde{C}_k(x + t,t)$ is a linear combination of the polynomials $Q_k(x,t)$ in (\ref{(tshumbral)}) when they are time-space harmonic with respect to the umbral counterpart of a compensated Poisson process.

Let $\{N_t\}_{t \geq 0}$ be a Poisson process of intensity $1$ and $X_t = N_t -t$ the compensated process \cite{schoutens2}. From the L\'evy-Khintchine formula \cite{sato}, its generating
function is $\phi(z,t)=\exp[t (e^z - z)].$ The umbra $\gamma$ having generating function $f(\gamma,z)=
e^z - z$ is the disjoint difference of the unity umbra $u$ and the singleton umbra $\chi.$ Indeed recall that
given two distinct umbrae $\delta_1$ and $\delta_2,$ their \emph{disjoint difference} is an auxiliary umbra, denoted by the symbol $\delta_1 \, \dot{-} \, \delta_2,$ with generating function $f(\delta_1,z) - f(\delta_2,z) + 1.$
Therefore, the umbral counterpart of a compensated Poisson process is $t \punt \beta \punt (u \dot{-} \chi).$ 
From Theorem \ref{AST_Levy1}, time-space harmonic polynomials with respect to a compensated Poisson process are $Q_k(x,t) = E[\{x-t \punt \beta \punt (u \dot{-} \chi)\}^k].$
\begin{prop}
We have $\widetilde{C}_k(x + t,t) = \sum_{j=1}^k s(k,j) Q_j(x,t),$ where $s(k,j)$ are the Stirling numbers of the first kind.
\end{prop}
\begin{proof}
By comparing generating functions, the sequence of polynomials $\{\widetilde{C}_k(x,t)\}_{k \geq 0}$ is umbrally represented
by the polynomial umbra $x \punt \chi - t.$ Therefore, we have $\widetilde{C}_k(x + t,t) =
E[((x+t) \punt \chi - t)^k].$ We also have $(x+t) \punt \chi - t \equiv (x \punt \chi - t) + t \punt \chi
\equiv ((x \punt \chi - t) \punt \beta + t) \punt \chi,$ due to the distributive property (\ref{(distributive1)}) and  $\beta \punt \chi \equiv u.$ Since $E[(\alpha \punt \chi)^k]=E[(\alpha)_k],$
see \cite{Dinardoeurop}, we have $\widetilde{C}_k(x + t,t) = E[((x \punt \chi - t) \punt \beta + t)_k]
= \sum_{j=0}^k s(k,j) E[((x \punt \chi - t) \punt \beta + t)^j].$ The result follows by showing that
$x-t \punt \beta \punt (u \dot{-} \chi) \equiv (x \punt \chi - t) \punt \beta + t.$ 
Indeed we have $x - t \punt \beta \punt (u \dot{-} \chi) \equiv  x - t \punt \beta - t \punt \beta \punt (-\chi) \equiv (x - t \punt \beta - t \punt \beta \punt (-\chi)) \punt \chi \punt \beta \equiv (x \punt \chi - t - t \punt \beta
\punt (-\chi) \punt \chi) \punt \beta,$ where we have used again $\chi \punt \beta \equiv u$ and the distributive property (\ref{(distributive1)}). The result follows since $- \chi \equiv \chi \punt (-1)$ and
$- t \punt \beta \punt (-\chi) \punt \chi \equiv - t \punt \beta \punt \chi
\punt (-1) \punt \chi \equiv t \punt \chi.$ 
\end{proof}
\subparagraph{L\'evy-Sheffer polynomials.}
According to the definition given in \cite{ST}, a sequence of polynomials $\{V_k(x,t)\}_{t \geq 0}$ is a
L\'evy-Sheffer system if it is defined by the following generating function 
\begin{equation}
\sum_{k \geq 0} V_k(x,t)\frac{z^k}{k!} = \left(g(z)\right)^t \exp\{x u(z)\},
\label{genfunlevshef}
\end{equation}
where
$g(z)$ and $u(z)$ are analytic in a neighborhood of $z = 0,$ $u(0) = 0,$ $g(0) = 1,$ $u'(0) \neq 0$ 
and $1/g(\tau(z))$ is an infinitely divisible moment generating function, with
$\tau(z)$ such that $\tau(u(z)) = z.$ Schoutens in \cite{ST} states that the basic link between 
these polynomials and L\'evy processes is a {\it martingale equality} (cf. pag. 337 eq. (6)), which 
is equivalent to ask that these polynomials are time-space harmonic with respect to
L\'evy processes. 

In the following, we show that $V_k(x,t)$ is a linear combination of suitable time-space harmonic polynomials
$Q_k(x,t)$ and therefore, they share the same property. 

To this aim we need to recall the notion of compositional inverse of an umbra $\alpha.$ 
Recall first that the auxiliary umbra $\alpha \punt \beta \punt \gamma$ is the \emph{composition umbra} of 
$\alpha$ and $\gamma,$ see \cite{Dinardoeurop}. The name recalls that
its generating function is the composition of $f(\alpha, z)$ and $f(\gamma, z),$
that is $f(\alpha \punt \beta \punt \gamma,z) = f(\alpha, f(\gamma, z) - 1).$ Moreover its
moments are 
\begin{equation}
E[(\alpha \punt \beta \punt \gamma)^i] = \sum_{k=1}^i a_k B_{i,k}(g_1, \ldots, g_{i-k+1}).
\label{(4c)}
\end{equation}
If $E[\alpha] \ne 0,$ the compositional inverse of an umbra $\alpha$ is the auxiliary umbra 
$\alpha^{{\scriptscriptstyle <-1>}}$ such that $\alpha \punt \beta \punt \alpha^{{\scriptscriptstyle <-1>}}
\equiv \alpha^{{\scriptscriptstyle <-1>}} \punt \beta \punt \alpha \equiv \chi,$ with $\chi$ the singleton
umbra. 
\begin{thm}
We have 
$V_k(x,t) = \sum_{i=0}^k E[(x + t \punt \beta \punt \kappa)^i] B_{k,i}(g_1, \ldots, g_{k-i+1}),$ where $g_j = E[\gamma^j],$ for all nonnegative $j$ and $\kappa$ is the cumulant umbra of $\alpha \punt \beta \punt \gamma^{{\scriptscriptstyle <-1>}},$ 
with $\gamma^{{\scriptscriptstyle <-1>}}$ the compositional inverse of the umbra $\gamma.$
\end{thm} 
\begin{proof}
Due to the form of the generating function in (\ref{genfunlevshef}), the polynomials $\{V_k(x,t)\}_{k \geq 0}$ are umbrally represented by the 
polynomial umbra $t \punt \alpha + x \punt \beta \punt \gamma,$ with $f(\alpha,z)= g(z)$ and $f(\gamma,z) = 1 + u(z),$
and $E[\gamma] = u^{\prime}(0) \ne 0.$  
Thanks to the distributive property (\ref{(distributive1)}) we have $x \punt \beta \punt \gamma + t \punt \alpha \equiv (x + t \punt \alpha \punt \beta \punt \gamma^{{\scriptscriptstyle <-1>}}) \punt \beta \punt \gamma.$ The result follows 
by using (\ref{(4c)}). 
\end{proof}
\begin{cor}
The L\'evy-Sheffer polynomials $\{V_k(x,t)\}_{t \geq 0}$ are time-space harmonic with respect to
L\'evy processes umbrally represented by   
$\{-t \punt \alpha \punt \beta \punt \gamma^{{\scriptscriptstyle <-1>}}\}_{t \geq 0}.$
\end{cor}
\subparagraph{Time-space harmonic polynomials in terms of boolean and free cumulants.}
Let $M(z)$ be the ordinary generating function of a random variable $X$,
that is $M(z)= 1 + \sum_{i \geq 1}  a_i z^i,$ where $a_i = E[X^i]$. We have
$M(z) = 1/[1-B(z)],$ where $B(z)= \sum_{i \geq 1} b_i
z^i,$ and $b_i$ are the boolean cumulants of $X$.
The noncrossing (or free) cumulants of $X$ are the
coefficients $r_i$ of the ordinary power series $R(z) = 1 +
\sum_{i \geq 1} r_i z^i$ such that $M(z)=R[z M(z)].$
The umbral theory of boolean and free cumulants has been introduced in \cite{free}.
In particular, the $\alpha$-boolean cumulant umbra $\eta_
{\scriptscriptstyle\alpha}$ has moments $E[\eta_
{\scriptscriptstyle\alpha}^i]=b_i$ for all nonnegative integers $i.$ This umbra is such that
$\bar{\alpha} \equiv \bar{u} \punt \beta \punt \bar{\eta}_{\scriptscriptstyle\alpha},$
where $E[\bar{u}^n] = n!$ and $E[\bar{\alpha}^n] = n! a_n$ for all nonnegative integers $n.$ Thanks to Theorem
\ref{UTSH2}, the polynomials $Q_k(x,t) = E[(x - t \punt \bar{u} \punt \beta \punt
\bar{\eta}_{\scriptscriptstyle\alpha})^k]$ are time-space harmonic polynomials
with respect to the family $\{ t \punt \bar{\alpha} \}_{t \geq 0}$
with $\eta_{\scriptscriptstyle\alpha}$ the $\alpha$-boolean cumulant umbra.
The $\bar{\alpha} \,$-free cumulant
$\mathfrak{K}_{\scriptscriptstyle \bar{\alpha}}$ has moments
$E[\mathfrak{K}_{\scriptscriptstyle
\bar{\alpha}}^i] = i! r_i,$ for all nonnegative integers $i.$ This umbra allows us
a different parametrization of $Q_k(x,t) = E[(x - t \punt \bar{\alpha})^k].$
Indeed, if we denote by
$\bar{\alpha}_{\scriptscriptstyle D}$ the derivative umbra of $\bar{\alpha},$ such that
$f(\bar{\alpha}_{\scriptscriptstyle D},z) = 1 + z f(\bar{\alpha},
z),$
then we have $\bar{\alpha} \equiv {\bar{\mathfrak{K}}}_{\scriptscriptstyle \alpha} \punt
\beta \punt {(-1 \punt {\bar{\mathfrak {K}}}_{\scriptscriptstyle
\alpha})_{\scriptscriptstyle D}^{\scriptscriptstyle <-1>}}.$
Therefore, also the polynomials $Q_k(x,t) = E[(x + t \punt (- 1 \punt
{\bar{\mathfrak{K}}}_{\scriptscriptstyle \alpha}) \punt
\beta \punt {(-1 \punt {\bar{\mathfrak {K}}}_{\scriptscriptstyle
\alpha})_{\scriptscriptstyle D}^{\scriptscriptstyle <-1>}})^k]$ are time-space harmonic polynomials
with respect to the family $\{ t \punt \bar{\alpha} \}_{t \geq 0}.$

\end{document}